\documentclass[reqno, 11pt]{amsart}

\makeatletter
\let\origsection=\section \def\section{\@ifstar{\origsection*}{\mysection}} 
\def\mysection{\@startsection{section}{1}\z@{.7\linespacing\@plus\linespacing}{.5\linespacing}{\normalfont\scshape\centering\S}}
\makeatother    

\usepackage{geometry}
\usepackage{ifthen}
\usepackage{tikz}
\usetikzlibrary{decorations.pathmorphing,decorations.shapes}

\tikzset{snake it/.style={decorate, decoration=snake}}
\tikzset{decorate sep/.style 2 args=
{decorate,decoration={shape backgrounds,shape=circle,shape size=#1,shape sep=#2}}}

\numberwithin{equation}{section}
\numberwithin{figure}{section}

\usepackage[ textwidth=25mm, textsize=scriptsize, color=green!20!white, linecolor=black!50!green]{todonotes}

\usepackage[backend=bibtex,style=nature,sorting=nty]{biblatex}
\usepackage{xcolor}
\usepackage{hyperref}
\hypersetup{
    unicode,
    colorlinks,
    linkcolor={red!60!black},
    citecolor={green!60!black},
    urlcolor={blue!60!black}
}

\usepackage{amsfonts,amsthm,amssymb,amsmath}
\usepackage[T1]{fontenc}
\usepackage{microtype}
\usepackage{graphicx}
\usepackage{tikz}
\usepackage[usenames,dvipsnames]{pstricks}
\usepackage{epsfig}
\usepackage{verbatim}
\usepackage{scalerel}
\usepackage[fleqn]{nccmath}

\usepackage{enumerate,enumitem}
\setlist[enumerate]{label=$(\arabic*)$}
\usepackage{thmtools, thm-restate}
\usepackage{tikz}
\usepackage{tikz-3dplot}
\usepackage{pgfplots}
\usetikzlibrary{calc,through,intersections,arrows, trees, positioning, decorations.pathmorphing, cd,decorations.pathreplacing,arrows.meta}

\usepackage{mathtools}
\usepackage{relsize}
\usepackage{subcaption}
\usepackage{chngcntr}
\counterwithout{figure}{section}

\newtheorem{theorem}{Theorem}
\numberwithin{theorem}{section}
\newtheorem{lemma}[theorem]{Lemma}  
\newtheorem{cor}[theorem]{Corollary}
\newtheorem{prop}[theorem]{Proposition}

\newtheorem{mainresult}{Theorem}
\newtheorem{ques}{Question}

\newtheorem*{Hconj}{Halin's conjecture}

\theoremstyle{definition}
\newtheorem{defn}[theorem]{Definition}

\newcommand{\script}{\mathcal}
\newcommand{\eps}{\varepsilon}
\newcommand{\parentheses}[1]{{\left( {#1} \right)}}

\newcommand{\p}{\parentheses}

\newcommand{\Set}[1]{{\left\lbrace {#1} \right\rbrace}}

\def\set#1:#2{\Set{{#1} \colon {#2}}}
\def\downcl#1{\lceil{#1}\rceil}
\def\upcl#1{\lfloor{#1}\rfloor}

\newcommand{\N}{\mathbb{N}}

\newcommand{\Pcal}{{\mathcal P}}

\renewcommand{\triangleleft}{\vartriangleleft}
\renewcommand{\leq}{\leqslant}
\renewcommand{\geq}{\geqslant}
\renewcommand{\ge}{\geq}
\renewcommand{\le}{\leq}

\renewcommand{\rho}{\varrho}
\renewcommand{\subset}{\subseteq}
\renewcommand{\supset}{\supseteq}
\newcommand{\nottriangleleft}{\not\kern-1pt\mathrel{\triangleleft}}
\newcommand{\rayinf}{\mathbin{\sharp}}

\newcommand{\weakch}{$\operatorname{GCH}$}

\newcommand{\cf}{\mathrm{cf}}
\newcommand{\ZFC}{\textnormal{ZFC}}
\newcommand{\CH}{\textnormal{CH}}
\newcommand{\GCH}{\textnormal{GCH}}

\DeclareMathOperator{\medcup}{\mathsmaller{\bigcup}}

\DeclareMathOperator{\HC}{\operatorname{HC}}

\begin{document}

\title{Halin's end degree conjecture} 

\author[Geschke]{Stefan Geschke}
\address{Universit\"at Hamburg, Department of Mathematics, Bundesstrasse 55 (Geomatikum), 20146 Hamburg, Germany}
\email{\{stefan.geschke, jan.kurkofka, ruben.melcher, max.pitz\}@uni-hamburg.de}

\author[Kurkofka]{Jan Kurkofka}
\author[Melcher]{Ruben Melcher}
\author[Pitz]{Max Pitz}



\begin{abstract}
An end of a graph $G$ is an equivalence class of rays, where two rays are equivalent if there are infinitely many vertex-disjoint paths between them in $G$.
The degree 
of an end is the maximum cardinality of a collection of pairwise disjoint rays in this equivalence class.

Halin conjectured that the end degree can be characterised in terms of certain typical ray configurations, which would generalise  his famous \emph{grid theorem}. In particular, every end of regular uncountable degree $\kappa$ would contain a \emph{star of rays}, i.e.\ a  configuration  consisting of  a central ray $R$ and $\kappa$ neighbouring rays $(R_i \colon i  < \kappa)$ all disjoint from each other and each $R_i$ sending a family of infinitely many disjoint paths to $R$ so that paths from distinct families only meet  in $R$.

We show that Halin's conjecture fails for end degree $ \aleph_1$,  holds for $\aleph_2,\aleph_3,\ldots,\aleph_\omega$, fails for $ \aleph_{\omega+1}$, and  is undecidable (in ZFC) for the next $\aleph_{\omega+n}$ with $n \in \N$, $n \geq 2$. Further results include a complete solution for all cardinals under GCH, complemented by a number of consistency results.
\end{abstract}

\vspace*{-1.14cm}
\maketitle

\vspace*{-.5cm}
\section{Overview}

\subsection{Halin's end degree conjecture} An \emph{end} of a graph $G$ is an equivalence class of rays, where two rays of $G$ are \emph{equivalent} if there are infinitely many vertex-disjoint paths between them in~$G$. 
The \emph{degree}   $\deg(\varepsilon)$
of an end $\varepsilon$ is the maximum cardinality of a collection of pairwise disjoint rays in $\varepsilon$, see~Halin~\cite{H65}.

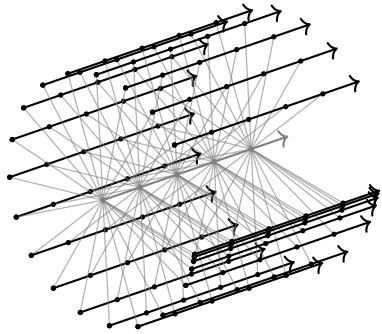
\begin{figure}[ht]
    \centering

\tdplotsetmaincoords{70}{135+12.25}
\begin{tikzpicture}[tdplot_main_coords = 1,scale=0.7]]
\tdplotsetrotatedcoords{12.5}{270}{320}
\pgfmathsetmacro\its{20}%
  
\foreach \n in {0,...,\its}
{

    \pgfmathparse{.75*\its}\edef\lif{\pgfmathresult}%
    \pgfmathparse{int(round(.5*\its))}\edef\flip{\pgfmathresult}%
    
    \pgfmathparse{int(round(\flip-1))}\edef\flup{\pgfmathresult}%
    
    "\ifthenelse{\n>\flup \AND \n<\lif}{\pgfmathparse{360*(\n/\its)}\edef\x{\pgfmathresult}}{\pgfmathparse{(45*(atan(\n-\lif))*(pi/180)/(pi/2))+270}\edef\x{\pgfmathresult}}"%
    "\ifthenelse{\n<\flip}{\pgfmathparse{360*(\flup-\n)/\its)}\edef\x{\pgfmathresult}}{}"%
    
    \draw[thick,<-,line cap=round,tdplot_rotated_coords] ({2.5*sin(\x )},{2.5*cos(\x )},5 ) --  ({2.5*sin(\x)},{2.5*cos(\x )},0);
    
    \foreach \y in {0,1, 2,3, 4} {
        \draw[gray, line width=.17mm,tdplot_rotated_coords,opacity=.6]
        ({2.5*sin( \x )},{2.5*cos( \x )},\y ) -- (0,0, \y);
        \fill[tdplot_rotated_coords] ({2.5*sin( \x )},{2.5*cos( \x )},\y ) circle (.5mm); 
    }

    \fill[tdplot_rotated_coords] ({2.5*sin( \x )},{2.5*cos( \x )},0 ) circle (.5mm);
}

\draw[thick,->,gray,tdplot_rotated_coords] (0,0,0) -- (0,0,5);
\foreach \y in {0,1, 2,3, 4} {
    \fill[gray,tdplot_rotated_coords] (0,0,\y ) circle (.5mm);
}
\end{tikzpicture}
\caption{The Cartesian product of a star and a ray.
}
    \label{fig:StarProductRay2}
\end{figure}

However, for many purposes a degree-witnessing collection $\script{R}\subset\eps$ on its own forgets significant information about the end, as it tells us nothing about how $G$ links up the rays in~$\script{R}$; in fact $G[\,\bigcup\script{R}\,]$  is usually disconnected.
Naturally, this raises the question of whether one can describe typical configurations in which $G$ must link up the disjoint rays in some degree-witnessing subset of a pre-specified end.

Observing  that prototypes of ends of any prescribed degree are given by the Cartesian product of a sufficiently large connected graph with a ray (see e.g.\ Figure~\ref{fig:StarProductRay2}), Halin \cite{H65} made this question precise by introducing the notion of a `ray graph', as follows. 

Given a set $\mathcal{R}$ of disjoint equivalent rays in a graph $G$, we call a graph~$H$ with vertex set $\mathcal{R}$ a \emph{ray graph} in $G$ if there exists a set $\Pcal$ of independent 
$\mathcal{R}$-paths (independent paths with precisely their endvertices on rays from $\mathcal{R}$)  in $G$ such that for each edge $RS$ of $H$ there are infinitely many disjoint $R$--$S$ paths in $\Pcal$. Given an end $\eps$ in  a graph $G$, a \emph{ray graph for $\eps$} 
is a connected ray graph in $G$ on a degree-witnessing subset of~$\eps$.
The precise formulation of the question reads as follows:

\begin{align*}
\textit{Does every graph contain ray graphs for all its ends?}
\end{align*}

For ends of finite degree it is straightforward to answer the question in the affirmative. For ends of countably infinite degree the answer is positive too, but only elaborate constructions are known. These constructions by Halin~\cite[Satz~4]{H65} and by Diestel~\cite{dlestel2004short,Bible} show that in this case the ray graph itself can always be chosen as a ray:

\begin{theorem}[Halin's grid theorem
]
\label{thm_halin}
Every graph with an end of infinite degree contains a subdivision of the hexagonal quarter grid whose rays belong to that end. 
\end{theorem}

For ends of uncountable degree, however, the question is a 20-year-old open conjecture that Halin stated in his legacy collection of problems~\cite{halin2000miscellaneous}:

\begin{Hconj}[{\cite[Conjecture~6.1]{halin2000miscellaneous}}]
Every graph contains ray graphs for all its ends.
\end{Hconj}

\noindent In this paper, we settle Halin's conjecture:
partly positively, partly negatively, with the answer essentially only depending on the degree of the end in question.

\subsection{Our results}  
If the degree  in question is $\aleph_1$, then any  ray graph for  such  an end contains a vertex of degree~$\aleph_1$, which together with its neighbours already forms a ray graph for the end  in question, namely an `$\aleph_1$-star of rays'.
Thus, finding in $G$ a  ray graph for an end of degree $\aleph_1$ reduces to finding such a star of rays.
Already this case has  remained  open.

Let $\HC(\kappa)$ be the statement that Halin's conjecture holds for all ends of degree $\kappa$ in any graph.
As our first main result, we show that

\begin{align*}
\HC(\aleph_1)\textit{ fails.}
\end{align*}

\noindent So, Halin's conjecture is not true after all.
But we do not stop here, for the question whether $\HC(\kappa)$ holds remains open for end degrees $\kappa>\aleph_1$.
And surprisingly, we show that $\HC(\aleph_2)$ holds.
In fact, we show more generally that
\begin{align*}
\HC(\aleph_n)\textit{ holds for all }n\textit{ with }2\le n\le\omega.
\end{align*}

\noindent Interestingly, this includes the first singular uncountable cardinal~$\aleph_\omega$.
Having established these results, it came as a surprise to us that
\begin{align*}
\HC(\aleph_{\omega+1})\textit{ fails.}
\end{align*}

\noindent How does this pattern continue? It turns out that from this point onward, set-theoretic considerations start playing a role. 
Indeed 
\begin{align*}
\HC(\aleph_{\omega+n})\textit{ is undecidable for all }n\textit{ with }2\le n\le\omega,\textit{while }\HC(\aleph_{\omega\cdot 2 +1})\textit{ fails.}
\end{align*}
\noindent The following theorem solves Halin's problem for all end degrees:

\begin{mainresult}\label{thm_mainresult} The following two assertions about $\HC(\kappa)$ are provable in \ZFC:
\begin{enumerate}[label=$(\arabic*)$]
\item\label{item1}   
$\HC(\aleph_n)$ holds for all $2 \leq n \leq \omega$, 
\item\label{item_ZFCcounterexample}  $\HC(\kappa)$ fails for all $\kappa$ with $\cf\p{\kappa} \in \set{\mu^+}:{\cf\p{\mu} = \omega}$; in particular, $\HC(\aleph_1)$ fails.
\end{enumerate}
Furthermore, the following assertions about $\HC(\kappa)$ are consistent:
\begin{enumerate}[label=$(\arabic*)$]
\setcounter{enumi}{2}
\item\label{item_GCHresults} Under \GCH, $\HC(\kappa)$ holds for all cardinals not excluded by \ref{item_ZFCcounterexample}.
\item\label{item_consistentcounterexamples} However, for all $\kappa$ with   $\cf\p{\kappa} \in \set{\aleph_{\alpha}}:{\omega<\alpha<\omega_1}$ it is consistent with $\ZFC {+} \CH$ that $\HC(\kappa)$ fails, and similarly also for all $\kappa$ strictly greater than 
the  least
cardinal $\mu$ with $\mu = \aleph_\mu$.
\end{enumerate}
\end{mainresult}

\noindent For the consistency results in \ref{item_consistentcounterexamples} we do not presuppose any advanced set-theoretic knowledge beyond the usual concepts of ordinals, order trees, cardinals and cofinality. Rather, we identify in the literature a suitable combinatorial statement about tree branches that is known to be consistent, and then set out from there to construct our counterexamples for Halin's conjecture in the cases of \ref{item_consistentcounterexamples}. The proof sketch below offers a flavour which statements we need  precisely.

The affirmative results in \ref{item1} and \ref{item_GCHresults} are proved in the first half of this paper, up to Section~\ref{sec_affirm}. 
The  counterexample for $\HC(\aleph_1)$, which can be read independently of the other chapters, is constructed in Section~\ref{sec_counter}. The remaining counterexamples  
in \ref{item_ZFCcounterexample} and \ref{item_consistentcounterexamples} are constructed in Sections~\ref{sec_lambdakappaII} to \ref{sec_lifting}.

\subsection{Proof sketch}

The first step behind our affirmative results for $\HC(\kappa)$ is the observation that it suffices to find some countable set of vertices $U$ for which there is a set $\mathcal{R}$ of  $\kappa$ many rays in $\varepsilon$, all disjoint from each other and from $U$, such that each $R \in \mathcal{R}$ comes with an infinite family $\mathcal{P}_R$ of disjoint $R$--$U$ paths which, for distinct $R$ and  $R'$, meet only in their endpoints in  $U$: then it is not difficult to find a ray $R^*$ that contains enough of $U$  to include the endpoints of almost all path families $\mathcal{P}_R$, yielding a $\kappa$-star of rays on $\Set{R^*} \cup \mathcal{R}'$ for some suitable $\mathcal{R}' \subset \mathcal{R}$ (Lemma~\ref{lem_countablecore}). 

While it may be hard to identify a countable such set $U$ directly, for  $\kappa$ of cofinality at least $\aleph_2$ there is a neat greedy approach to finding a similar set $U'$ of cardinality just $<\kappa$ rather than $\aleph_0$ (Lemma~\ref{lem_greedy}). 

Let us illustrate the idea in the case of $\kappa  = \aleph_2$: Starting from an arbitrary ray $R_0$ in  $\eps$, does $U_0 = V(R_0)$ already do the job? 
That is to say,  are there $\aleph_2$    disjoint rays in $\eps$ that are independently attached to $U_0$ as above? 
If so, we have achieved our goal. If not, take a maximal set of  disjoint rays $\mathcal{R}_0$ in $\eps$ all whose rays are independently attached to $U_0$ as above, and define  $U_1$ to be the union of $U_0$ together with the vertices from all the  rays in  $\mathcal{R}_0$ and all their selected paths to $U_0$. Then $|U_1| \leq \aleph_1$. Does $U_1$ do the job? If so, we have achieved our goal. If not, continue as above. 
We claim that when continuing transfinitely and building sets $U_0 \subsetneq U_1 \subsetneq \ldots \subsetneq U_\omega \subsetneq U_{\omega+1} \subsetneq \ldots $, we will achieve our goal at some countable ordinal~$< \omega_1$.
For suppose not. Then $U'' := \bigcup \set{U_i}:{i < \omega_1}$ meets all the rays in $\varepsilon$. Indeed, any ray $R$ from  $\eps$  outside of $U''$ could be joined to $U''$ by an infinite family of  disjoint $R$--$U''$ paths. But then their countably many endvertices already belong to some $U_i$  for $i< \omega_1$, contradicting the maximality of $\mathcal{R}_{i}$ in the definition of $U_{i+1}$. Hence, $|U''| = \aleph_2$. For  cofinality reasons there is a  first index $j = i+1$ with $|U_j| = \aleph_2$. Now $U'=U_i$ is as required.

Having identified a ${<}\kappa$-sized set $U'$ together with $\kappa$ disjoint rays all independently attached to it, we aim  to restrict $U'$ to a countable set $U$ while keeping $\kappa$ many rays attached to $U$ (Section~\ref{sec_lambdakappa}). For $\kappa=\aleph_2$ this is straightforward, since if $U'$ is written as an increasing $\aleph_1$-union of countable sets, one of them already contains all the endpoints of the path systems for some $\aleph_2$-sized subcollection $\mathcal{R}' \subset \mathcal{R} $.
Take this countable set as the set $U$ originally sought. This completes the proof of $\HC(\aleph_2)$; the  other affirmative results for $\HC(\kappa)$ are similar (Corollary~\ref{cor_citeintro} and Theorem~\ref{thm_cof>omega1case}).

What about general cardinalities $\kappa$? The above strategy can fail in two different ways: First, if $\cf\p{\kappa} = \aleph_1$, the greedy approach may not terminate: for example, it may well be possible that $|U'| = \aleph_1$ while $|U_i| = \aleph_0$ for all $i < \omega_1$. 
And indeed, we will show that rays in ends of degree $\aleph_1$ may be `arranged like an Aronszajn tree',  witnessing  the failure of $\HC(\aleph_1)$. 
This idea can be captured as follows (see Definition~\ref{def_rayinflation} for precise details): 
For an Aronszajn tree  $T$, consider first a disjoint family of rays  $\set{R_t}:{t  \in T}$ indexed by the nodes of the tree. 
If $t$ is a successor of s in $T$, add an infinite matching between the rays $R_t$ and $R_{s}$. 
And if $t$ is a limit, pick a cofinal $\omega$-sequence 
$t_0 < t_1 < \ldots <  t$ of nodes below $t$, and add an  edge from the $n$th vertex of $R_t$ to the  $n$th vertex of $R_{t_n}$, for all $n \in \N$. 
If these cofinal sequences $t_0 < t_1 < \ldots <  t$ below each limit $t$ are chosen carefully
(for this we rely on a trick by Diestel, Leader and Todor\v{c}evi\'c, see~Theorem~\ref{thm_dl_Aronszajn_with_property}), 
the resulting graph, which we call the \emph{ray inflation of $T$}, is one-ended of degree~$\aleph_1$ but contains no $\aleph_1$-star of rays. This  refutes  $\HC(\aleph_1)$. See~Theorem~\ref{thm_ATcounterexample} for details.

What about the remaining cardinals $\kappa$ with $\cf\p{\kappa} = \aleph_1$? Also there, counterexamples to $HC(\kappa)$ exist, and we have a machinery that produces a multitude of such examples: 
Any counterexample for  $\HC(\kappa)$ for regular $\kappa$ may be turned canonically into a counterexample for  $\HC(\lambda)$ for all $\lambda$ with $\cf\p{\lambda} = \kappa$, see Theorem~\ref{thm_cofinalityupwards}.

The second way in which our above strategy can fail is that even if our greedy algorithm terminates and provides a ${<} \kappa$-sized $U'$ to which there are  $\kappa$  disjoint rays  independently attached, it may not be possible to further reduce $U'$ as earlier to some countable subset $U$. 
And indeed, using our idea of ray inflations of order trees, also this constellation can be exploited to construct counterexamples to Halin's conjecture. However, the trees that work now are quite different from the earlier Aronszajn trees: Generalising the concept of \emph{binary trees with tops} introduced by Diestel and Leader in \cite{DiestelLeaderNST}, we consider the class of \emph{$\lambda$-regular trees with tops}, where $\lambda$ is any singular cardinal of countable cofinality. 
These are order trees of height $\omega+1$ in which every point of finite height has exactly $\lambda$ successors, and above some $\kappa> \lambda$ many selected branches we add further nodes to the tree at height $\omega$, called \emph{tops}.

There is a reason why we take $\lambda$ to be singular  of countable cofinality: Just like the binary tree has uncountably many branches, these $\lambda$'s are the only other cardinals for which an uncountable regular tree is guaranteed to have strictly more than $\lambda$ branches. And just like the precise number of branches of the binary tree is not determined in ZFC alone (it is $2^{\aleph_0}$, which may be  $\aleph_1$ if \CH\ holds, or may be arbitrarily large), also the precise number of branches of the $\lambda$-regular tree is $\lambda^+$ if \GCH\ holds, but it also  may be  much larger.

Now the starting point for our consistent counterexamples of \ref{item_consistentcounterexamples} in Theorem~\ref{thm_mainresult} are simply models of $\ZFC + \CH$ in which the two $\lambda$-regular trees for $\lambda = \aleph_\omega$ and $\lambda$ equal to the first fixed point of the $\aleph$-function have a lot more branches than nodes. In these cases, any $\lambda$-regular tree with tops gives rise to counterexamples for  Halin's conjecture (Theorem~\ref{thm_scalecounterexamples}). What happens if one looks for \ZFC-counterexamples, not just consistent ones? With significantly more effort, and building on the concept of a \emph{scale} from Shelah's pcf-theory, we will show in Theorem~\ref{thm_MaxRewriteStefan} that for any singular $\lambda$ of countable cofinality one can directly select a suitable set of $\lambda^+$ many branches so that the $\lambda$-regular tree with corresponding tops gives rise to the counterexamples for Halin's conjecture, settling the remaining cases of \ref{item_ZFCcounterexample} in Theorem~\ref{thm_mainresult}.

\subsection{Open problems}

We suspect that \ref{item1} and \ref{item_ZFCcounterexample} in Theorem~\ref{thm_mainresult}  capture  all the cases of Halin's conjecture  that can be proved or disproved in \ZFC\ alone. This is certainly true up to $\aleph_{\omega_1}$, as for each $\kappa \leq \aleph_{\omega_1}$ our main Theorem~\ref{thm_mainresult} provides either a ZFC or an independence result regarding the  truth value of $\HC(\kappa)$. While for all remaining cardinals assertion~\ref{item_GCHresults} of Theorem~\ref{thm_mainresult} gives consistent affirmative results, we do not know whether any of these  can be established in ZFC.

\begin{ques} 
Is Halin's conjecture true for any $\kappa > \aleph_{\omega_1}$? 
\end{ques}
\begin{ques}
Is it true that any end of degree $\aleph_1$ 
either contains an $\aleph_1$-star of rays or a subdivision of a ray inflation of an Aronszajn tree, as  in Theorem~\ref{thm_ATcounterexample}?
\end{ques}

\begin{ques} Is it true  that for every cardinal $\kappa$ there is $f(\kappa) \geq  \kappa$, such that every end $\eps$ of degree $ f(\kappa)$ contains a connected ray graph of size $\kappa$?
\end{ques}

\section{Ray collections with small core}

For general notions in graph theory and in set theory we generally follow the textbooks by Diestel \cite{Bible} and Jech \cite{jech2013set}.
Let $\eps$ be an end of a graph $G$, and $U$ be a set of vertices in $G$. 
An $\eps$--$U$ \emph{comb} is a subgraph $C = R \cup \bigcup \script{P}$ of $G$ that consist of a ray $R$ disjoint from $U$ that represents~$\eps$ and an infinite family~$\script{P}$ of  disjoint $R$--$U$ paths. The vertices in $C \cap U$ are the \emph{teeth} of the comb. We write $\mathring{C}$ for $C - U$, the \emph{interior} of the $\eps$--$U$ comb, which is  
disjoint from $U$.
We call two $\eps$--$U$ combs \emph{internally disjoint} if they have disjoint interior.

\begin{lemma}
\label{lem_countablecore}
Let $\eps$ be an end of a graph $G$ and $U$ a countable set of vertices. 
If there is an uncountable collection $\mathcal{C}$ of internally disjoint $\eps$--$U$ combs in $G$, then $\eps$ contains a $|\mathcal{C}|$-star of rays whose leaf rays are the spines of (a subset of) combs in $\mathcal{C}$.
\end{lemma}

As this lemma  is fundamental  to our  affirmative results, we provide two proofs:  one relying on the theory of  normal trees, and a second, more elementary proof, using dominating vertices. 
Recall that a rooted tree $T \subset G$ is \emph{normal (in $G$)} if the endvertices of any $T$-path in $G$ (a non-trivial path in $G$ with endvertices in $T$ but all edges and inner vertices outside of $T$) are comparable in the tree order of $T$. 
Thus, if $T\subset G$ is normal, then the neighbourhoods of the components of $G-T$ form chains in~$T$.
The rays of $T$ starting in the root are called  \emph{normal rays}.

\begin{proof}[First proof of Lemma~\ref{lem_countablecore}]
We may assume that $G$ is connected.
Since $U$ is countable, by Jung's Theorem (see \cite[Satz~6]{jung1969wurzelbaume}, or the proof in \cite[Theorem~8.2.4]{Bible}), there is a countable normal tree $T\subset G$ that includes $U$. As $T$ is countable,  $\mathcal{C}$~contains a subcollection $\mathcal{C}'$  of size $\vert \mathcal{C} \vert$ such that the interior of every comb in $\mathcal{C}'$ is disjoint from $T$. 
As $T$ is normal, the teeth of any comb in $\mathcal{C}'$ lie on the unique normal ray $R$ in $T$ that represents $\eps$, see \cite[Lemma~8.2.3]{Bible}. 
Consequently, $R$ with all combs in $\mathcal{C}'$ forms the desired $\vert\mathcal{C}\vert$-star of rays.
\end{proof}

A vertex $v\in G$ \emph{dominates} a ray $R\subset G$ if $v$ is the centre of a subdivided infinite star with all its leaves in~$R$.
It \emph{dominates} an end $\eps$ of $G$ if it dominates some (equivalently:~each) ray in~$\eps$.

\begin{proof}[Second proof of Lemma~\ref{lem_countablecore}]
Denote by $U' \subset U$  the vertices of  $U$ that are  teeth of only finitely many combs in $\mathcal{C}$, and let $\mathcal{C}'\subset\mathcal{C}$ be the subcollection of combs with a  tooth in $U'$. 
Then $\mathcal{C}'$ is countable, so $\mathcal{C} \setminus \mathcal{C}'$  is of size  $\vert \mathcal{C} \vert$, and every comb in $\mathcal{C} \setminus \mathcal{C}'$ has all its teeth contained in~$U\setminus U'$. 
Since all vertices in $U \setminus U'$ dominate the  end $\eps$, there is a ray $R$ in $G$  with $U \setminus U' \subset V(R)$, cf.~\cite[Ex.~8.29 \& 8.30]{Bible}. Since $R$ is countable,  $\mathcal{C} \setminus \mathcal{C}'$ contains a subcollection $\mathcal{C}''$  of size $\vert \mathcal{C} \vert$ such that the interior of every comb in $\mathcal{C}''$ is disjoint from $R$. Consequently, $R$~together with all the combs in $\mathcal{C}''$ forms the desired $\vert\mathcal{C}\vert$-star of rays.
\end{proof}

Given an end $\eps$ of a graph $G$ with uncountable degree~$\kappa$, finding a countable vertex set $U\subset V(G)$ so that $G$ admits $\kappa$ internally disjoint $\eps$--$U$ combs (i.e.\ so that $U$ satisfies the premise of Lemma~\ref{lem_countablecore}) may be hard, and sometimes even impossible by Theorem~\ref{thm_mainresult}\ref{item1}.
Perhaps surprisingly, $U$ can be found greedily if $\kappa$ has large cofinality and the condition $\vert U\vert=\aleph_0$ is relaxed to $\vert U\vert<\kappa$.

\begin{lemma}[The greedy lemma]
\label{lem_greedy}
Let $\mathcal{R}$ be any $\kappa$-sized collection of disjoint rays belonging to  an end~$\eps$ of~$G$.  
If  $\cf\p{\kappa} > \aleph_1$, then there exist a set of vertices $U$ in $G$ with $|U| < \kappa$ and a $\kappa$-sized collection $\mathcal{C}$ of internally disjoint $\eps$--$U$ combs in $G$ with all spines in $\mathcal{R}$.
\end{lemma}

\begin{proof}
Let $U_0 := V(R)$ for an arbitrarily chosen ray $R \in \mathcal{R}$. 
Recursively construct a sequence $(U_i \colon i < \omega_1)$ of vertex sets of $G$ as follows: 
If $U_i$ is already defined, use Zorn's lemma to choose a maximal collection $\mathcal{C}_i$ of internally disjoint $\varepsilon$--$U_i$ combs with spines in $\mathcal{R}$, and let $U_{i+1} := U_i \cup \bigcup \mathcal{C}_i$. 
For a limit $\ell < \omega_1$, simply define $U_\ell := \bigcup_{i < \ell} U_i$.

Consider $U' := \bigcup_{i < \omega_1} U_i$. 
If $|U'| < \kappa$, then since $|\mathcal{R}| = \kappa$ there still exists an $\varepsilon$--$U'$ comb $C$ in $G$ with spine in $\mathcal{R}$. However, the countably many teeth from $C \cap U'$ belong already to some $U_i$ for $i < \omega_1$. But then the existence of $C$ contradicts the maximality of $\mathcal{C}_i$.

Hence,  $|U'| = \kappa$, and $\cf\p{\kappa} > \omega_1$ implies that there is a first $i < \omega_1$ such that $|U_i| = \kappa$ which must be a successor, say $i = j+1$. Then $U := U_j$ satisfies $|U| < \kappa$ and $\mathcal{C}:=\mathcal{C}_j$ is a $\kappa$-sized collection of $\eps$--$U$ combs as desired.
\end{proof}

For which cardinals $\kappa$ it is possible to bridge the gap between Lemma~\ref{lem_countablecore} and Lemma~\ref{lem_greedy} is discussed in Section~\ref{sec_lambdakappa} below. Before doing so, however, we use  Lemma~\ref{lem_countablecore} to  classify the minimal types of connected ray graphs in the next section.

\section{Typical types of ray graphs}

It is well known that every countable connected graph contains  a vertex of infinite degree or a ray. This carries over to larger cardinals as follows. Every connected graph on $\kappa$ many vertices, for regular uncountable $\kappa$, has a vertex of degree $\kappa$.
Indeed, consider the distance classes from any fixed vertex of the graph. Then there is a first distance class with $\kappa$ many vertices and the prior distance class contains a vertex of degree $\kappa$. We call a star with $\kappa$ many leafs a \emph{$\kappa$-star.}

A \emph{frayed  star} is a rooted tree in which all vertices are within distance 2 from the root.
The root is also called the \emph{centre} and the neighbours of the root are called the \emph{distributor vertices} of the frayed star. 
For a cardinal~$\kappa$ and a cofinal sequence $s=(\kappa_i\colon i<\cf (\kappa))$ in $\kappa$, a $(\kappa, s)$-\emph{star} is a frayed star such that the sequence of degrees of the distributor vertices corresponds to $s$. 
Note that if for a frayed star $S$ with $\kappa$ many leafs the  degrees of the distributor vertices form a cofinal sequence for $\kappa$, then $S$ contains a $(\kappa , s)$-star for any prescribed  $\cf(\kappa)$-sequence $s$ of $\kappa$ as a subgraph. 
A \emph{frayed comb} is a ray together with infinitely many  disjoint stars such that exactly one vertex of every star lies on the ray. The centres of these stars are the \emph{distributor vertices} and the leafs of the stars that do not lie on the ray are the \emph{teeth} of the frayed comb.

Recall that every connected graph on $\kappa$ many vertices contains a subdivided frayed star or a subdivided frayed comb with $\kappa$ many leafs or teeth, respectively. To see this we may assume the graph to be a tree $T$ rooted at $r$. Now, consider the size of every component of $T-r$. If there is one component $C$ of size $\kappa$ we view $C$ as a tree rooted at the neighbour of $r$ and continue in~$C$. If there is again and again a component of size $\kappa$ it is straightforward to find a frayed comb with $\kappa$ many teeth.  However, if all components are smaller than $\kappa$ and their size forms a cofinal sequence of $\kappa$ it is straightforward to find a subdivided frayed star with $\kappa$ many leafs. Finally, if all components are smaller than $\kappa$ and their size does not form a cofinal sequence in $\kappa$, there are at least $\kappa$ many components in which case we even find a star of size $\kappa$. See also \cite[Corollary~8.1]{gollin2018characterising} for a strengthening of the above results.

We now lift these basic facts to ray graphs. For  connected graph $H$, we say that the union of  disjoint rays $(R_h\colon h \in V(H))$ in some graph $G$ together with a path family $\mathcal{P}$ in $G$ is a \emph{$H$ graph of rays} in $G$ if the ray graph of $(R_h\colon h \in V(H))$ and $\mathcal{P}$ is $H$. If $X \subseteq V(H)$ is a set of vertices we say that the rays $(R_x \colon x \in X)$ are the \emph{$X$ rays} in $G$. If $H=S$ is a star with centre $c \in V(S)$ and leafs~$L$, then we say that $R_c$ is the centre ray and $(R_\ell \colon \ell \in L)$ are the leaf rays.

\begin{lemma}\label{lem_fundamentalRayGraphs}
Let $\eps$ be any end of a graph~$G$ of degree $\kappa$, and suppose that $\HC(\kappa)$ holds. 
\begin{enumerate}
    \item If $\kappa$ is regular and uncountable, then $G$ contains a $\kappa$-star of rays all belonging to~$\eps$.
    \item If $\kappa$ is singular and $s$ is any $\cf(\kappa)$-sequence of cardinals with supremum~$\kappa$, then $G$ contains either a $\kappa$-star of rays all belonging to~$\eps$ or a $(\kappa,s)$-star of rays all belonging to~$\eps$.
\end{enumerate}
\end{lemma}

\begin{proof}
As $\kappa$ is regular in assertion (1), any ray graph witnessing that $\HC(\kappa)$ holds is connected, and so has a vertex of degree $\kappa$. Clearly, any such vertex gives  rise to a $\kappa$-star of rays in $G$. 

For (2) we find either a frayed star or a frayed comb with $\kappa$ many leafs or teeth, respectively, in the connected ray graph. First, suppose that we obtain a subdivided  frayed star with $\kappa$ many leafs. This gives rise to a subdivided  frayed star of rays in $G$.  However, this can easily be turned into a $\kappa$-star of rays or a $(\kappa,s)$-star of rays, respectively. Indeed, any ray which has degree two in the ray graph might be abandoned to find a path family between its two neighbour rays (internally disjoint from all other rays and paths).
Second, suppose that we obtain a frayed comb of size $\kappa$ in the ray graph. Then there is a frayed comb of rays in $G$ and by a similar argument as above we may assume that the frayed comb has no vertices of degree two. Hence, there are only countably many vertices in the frayed comb that are not a tooth; denote by $U$ the set of all these vertices. Now, apply Lemma~\ref{lem_countablecore} to the set $ U'=\{V(R_u) \colon u \in U \}$  and the $\varepsilon$--$U'$ combs provided by the teeth rays to obtain a $\kappa$-star of rays.
\end{proof}

\section{On \texorpdfstring{$(\lambda,\kappa)$}{(lambda,kappa)}-graphs I}
\label{sec_lambdakappa}

A \emph{$(\lambda,\kappa)$-graph} is a bipartite graph $(A,B)$ with $|A| = \lambda$ and $|B| = \kappa$, where $\lambda<\kappa$ are infinite cardinals and every vertex $b \in B$ has infinitely many neighbours in $A$.

Next to their role in this paper,  $(\lambda,\kappa)$-graphs and in  particular  $(\lambda,\lambda^+)$-graphs also occur in the characterisation of the Erd\H{o}s-Hajnal colouring number \cite{bowler2019colouring} as well as in the characterisation of graphs with normal spanning trees \cite{pitz2020d}.
For a structural classification of $(\aleph_0,\aleph_1)$-graphs see \cite{BowlerGeschkePitzNST}.

In the following we investigate the question  which $(\lambda,\kappa)$-graphs contain an $(\aleph_0,\kappa)$-subgraph. 
Whenever  we speak of a $(\lambda',\kappa)$-subgraph of another $(\lambda,\kappa)$-graph in this paper, we silently  assume that the bipartition classes are respected. This is no significant restriction, since whenever a $(\lambda,\kappa)$-graph $(A,B)$ has a $(\lambda',\kappa)$-subgraph $(C,D)$, then it also has such a subgraph that respects the bipartition classes (simply restrict to the vertices in $D \setminus A$ and  their neighbours in $A \cap  C$).

\begin{lemma}
\label{lem_reducing}
Any $(\lambda,\kappa)$-graph with $\cf(\lambda)> \omega$ and $\cf\p{\kappa} \neq \cf\p{\lambda}$ contains a  $(\lambda',\kappa)$-subgraph for some $\lambda' < \lambda$.
\end{lemma}

\begin{proof}
Let $(A,B)$  be a  $(\lambda,\kappa)$-graph as in  the lemma. We  may assume that  $N(b) \subseteq A$ is countable for all $b  \in B$. 
Write $A =\bigcup \set{A_i}:{i < \cf\p{\lambda}  }$ as an increasing union of subsets $A_i$ with $|A_i| <  \lambda$. Let $B_i :=\set{b\in B}:{N(b)  \subset A_i}$. Since $\cf(\lambda)> \omega$,  we  have $B =\bigcup \set{B_i}:{i < \cf\p{\lambda}  }$. 
Recall that if a non-decreasing $\gamma$-sequence in a limit ordinal $\alpha$ is cofinal in $\alpha$, then $\cf(\gamma)=\cf(\alpha)$, see  e.g.\ \cite[Lemma~3.7(ii)]{jech2013set}.
As $\cf\p{\kappa} \neq \cf\p{\lambda}$, we therefore have $|B_i| = |B|$  for some $i < \cf\p{\lambda}$. Then $(A_i,B_i)$ is the desired $(\lambda',\kappa)$-subgraph.
\end{proof}

\begin{cor}
\label{cor_aleph_0kappa}
Any $(\lambda,\kappa)$-graph with $ \kappa  \leq \aleph_\omega$ contains an $(\aleph_0,\kappa)$-subgraph. \qed
\end{cor}

However, this pattern breaks down at the cardinal $\aleph_{\omega+1}$. In fact, we will show below that  for every singular cardinal $\lambda$ of countable cofinality, there exist $(\lambda,\lambda^+)$-graphs without $(\aleph_0,\lambda^+)$-subgraph, Theorem~\ref{thm_MaxRewriteStefan}.

In order to prove the affirmative results of item \ref{item_GCHresults} in our  main result Theorem~\ref{thm_mainresult}, we now show  that under \GCH\footnote{In fact, a close inspection shows that our results only require the consequence of GCH that $\lambda^{\aleph_0} =  \lambda^+$ for all singular cardinals $\lambda$ of countable cofinality.}, these are more or less the only exceptions:

\begin{lemma}
\label{lem_GCH}
Under \weakch, every $(\lambda,\kappa)$-graph with $\cf\p{\kappa} \notin \set{\mu^+}:{\cf\p{\mu} = \omega}$ and $\cf\p{\kappa} > \lambda >\aleph_0$ contains a  $(\lambda',\kappa)$-subgraph for some $\lambda' < \lambda$.
\end{lemma}

\begin{proof}
Let $(A,B)$  be a  $(\lambda,\kappa)$-graph as in  the lemma. We  may assume that  $N(b) \subseteq A$ is countable for all $b  \in B$. 
If $\cf\p{\lambda} > \omega$,  the assertion follows from Lemma~\ref{lem_reducing}.

If $\cf\p{\lambda} = \omega$, \weakch\ implies that  $|[A]^{\aleph_0}| =  \lambda^+$, see \cite[Theorem~5.15(ii)]{jech2013set}.
By assumption, $\cf\p{\kappa} > \lambda$  and $\cf\p{\kappa} \neq \lambda^+$, so $\cf\p{\kappa} > \lambda^+$. Hence, for some countable subset $A' \subseteq A$ we find a $\kappa$-sized $B' \subseteq B$ such that $N(b) = A'$  for all $b \in B'$. Then $(A',B')$ is the  desired subgraph.
\end{proof}

\begin{cor}
\label{cor_aleph_0regularkappa}
Under \weakch, every $(\lambda,\kappa)$-graph with $\cf\p{\kappa} \notin \set{\mu^+}:{\cf\p{\mu} = \omega}$ and $\cf\p{\kappa} > \lambda$ contains an $(\aleph_0,\kappa)$-subgraph.
\qed
\end{cor}

\begin{cor}
\label{cor_aleph_0regular}
Under \weakch, every $(\lambda,\kappa)$-graph for regular $\kappa \notin \set{\mu^+}:{\cf\p{\mu} = \omega}$ contains an $(\aleph_0,\kappa)$-subgraph.
\qed
\end{cor}

\begin{lemma}
\label{lem_final}
Under \weakch, every $(\lambda,\kappa)$-graph with $\cf\p{\kappa} \notin \set{\mu^+}:{\cf\p{\mu} = \omega}$ contains an $(\aleph_0,\kappa)$-subgraph or otherwise a collection of disjoint $(\aleph_0,\kappa_i)$-subgraphs for $\set{\kappa_i}:{i < \cf\p{\kappa}}$ cofinal in $\kappa$.
\end{lemma}

\begin{proof}
Let $(A,B)$ be some $(\lambda,\kappa)$-graph as in the statement of the lemma, but without an $(\aleph_0,\kappa)$-subgraph. 
By Corollary~\ref{cor_aleph_0regular}, $\kappa$ is singular and we may write $\kappa = \sup \set{\kappa_i}:{i <  \cf\p{\kappa}}$ for regular $\kappa_i > \lambda$ with $\kappa_i \notin \set{\mu^+}:{\cf\p{\mu} = \omega}$. Recursively choose disjoint $(\aleph_0,\kappa_i)$-subgraphs $(A_i,B_i)$ of $(A,B)$ for $i < \cf\p{\kappa}$ as follows: Given some fixed  $i < \cf\p{\kappa}$ with pairwise disjoint $(A_j,B_j)$ already selected for all $j < i$, let $$A'_i := A \setminus \medcup \set{A_j}:{j < i} \; \text{ and } \; B'_i := \set{b \in B \setminus \medcup \set{B_j}:{j < i}}:{|N(b) \cap A'_i| = \infty}.$$ 
If $|B'_i| < \kappa$, then since $|\bigcup \set{A_j}:{j < i} | < \cf\p{\kappa}$, the subgraph 
$$\p{\medcup \set{A_j}:{j < i} ,B \setminus B'_i}$$
of $(A,B)$ would contain an $(\aleph_0,\kappa)$-subgraph by Corollary~\ref{cor_aleph_0regularkappa}, contradicting our initial assumption. Therefore, we may choose $B''_i \subset B'_i$ of size $\kappa_i$, and apply Corollary~\ref{cor_aleph_0regular} to the graph $(A'_i,B''_i)$ and obtain an $(\aleph_0,\kappa_i)$-subgraph $(A_i,B_i)$ of $(A,B)$ disjoint from all  $(A_j,B_j)$ for $j < i$ as desired. 
\end{proof}

\begin{cor}
\label{rubens_cfomega_case}
Under \weakch, every $(\lambda,\kappa)$-graph with $\cf(\kappa) =\omega$   contains  an $(\aleph_0,\kappa)$-subgraph. 
\end{cor}

\begin{proof}
By Lemma~\ref{lem_final}, any $(\lambda,\kappa)$-graph with $\cf\p{\kappa} = \omega$ contains either an $(\aleph_0,\kappa)$-subgraph, in which case we are done, or otherwise a collection of disjoint $(\aleph_0,\kappa_i)$-subgraphs $(A_i,B_i)$ for $\set{\kappa_i}:{i < \omega}$ cofinal in $\kappa$. But then $(\bigcup_{i <  \omega} A_i,\bigcup_{i <  \omega} B_i)$ is an $(\aleph_0,\kappa)$-subgraph. 
\end{proof}

\section{Affirmative cases in Halin's conjecture}
\label{sec_affirm}

\subsection{Regular cardinals} Our  affirmative results for regular cardinals are now a straightforward consequence of Lemmas~\ref{lem_countablecore} and \ref{lem_greedy},  together with the results from the  previous section.

\begin{prop}
\label{thm_aleph_n}
Let $\mathcal{R}$ be  any $\kappa$-sized collection of disjoint equivalent rays in a graph~$G$ for an uncountable regular cardinal~$\kappa$. If $\kappa = \aleph_n$ for some $n \in \N$ with $n \geq 2$, then $G$ contains a $\kappa$-star of rays with leaf rays in $\mathcal{R}$. 

Assuming \weakch, there exists a $\kappa$-star of rays in $G$ with leaf rays in $\mathcal{R}$ unless $\kappa \in \set{\mu^+}:{\cf\p{\mu} = \omega}$.
\end{prop}

\begin{proof}
Let $\mathcal{R}$ be  any collection of disjoint equivalent rays of regular cardinality $\kappa > \aleph_1$ belonging to some end $\eps$ of $G$. By Lemma~\ref{lem_greedy}, there exist a set of vertices $U$ in $G$ with $\lambda:= |U| < \kappa$ and a $\kappa$-sized collection $\mathcal{C}$ of internally disjoint $\eps$--$U$ combs in $G$ with spines in $\mathcal{R}$. 

Consider the $(\lambda,\kappa)$-minor $H=(U,\mathcal{C})$ of $G$ where we contract the interior of every $\eps$--$U$ comb  $C \in \mathcal{C}$. Applying Corollary~\ref{cor_aleph_0kappa} in the case where $\kappa = \aleph_n$ for some $n \in \N$ with $n \geq 2$, and Corollary~\ref{cor_aleph_0regularkappa} otherwise, we obtain that $H$ contains an $(\aleph_0,\kappa)$-subgraph $H'$. 
After uncontracting the combs on the  $\kappa$-side of~$H'$, applying Lemma~\ref{lem_countablecore} finishes the proof.
\end{proof}

\begin{cor}
\label{cor_citeintro}
$\HC(\aleph_n)$ holds for all $n \in \N$ with $n \geq 2$. 
Moreover, under \weakch, $\HC(\kappa)$ holds for all regular cardinals $\kappa \notin \set{\mu^+}:{\cf\p{\mu} = \omega}$. \qed
\end{cor}

\subsection{Singular cardinals} We now extend these affirmative results to singular cardinals. First, to singular cardinals of countable cofinality, and then to all singular  cardinals whose cofinality is not a successor of a regular cardinal. By case \ref{item_ZFCcounterexample} of our main Theorem~\ref{thm_mainresult}, this is best possible.

\begin{prop}
\label{thm_cofomegacase}
$\HC(\aleph_\omega)$ holds. Moreover, under \weakch, $\HC(\kappa)$ holds whenever $\cf\p{\kappa} = \omega$.
\end{prop}

\begin{proof}
Suppose $\kappa > \cf\p{\kappa} = \omega$, and let
$(\kappa_n\colon n\in\omega)$ be a strictly increasing sequence of infinite cardinals with supremum $\kappa$, and each of the form $\kappa_n = \lambda_n^{++}$. 
Let $G$ be any graph and $\varepsilon$ an end of $G$ of degree $\kappa$. For each $n\in\omega$ we use Proposition~\ref{thm_aleph_n} to find a $\kappa_n$-star of rays~$S_n$ in $G$ with all rays in~$\eps$.
We write $R_n$ for the centre ray of~$S_n$ and consider the countable vertex set $U:=\bigcup_{n\in\omega}V(R_n)$.
For each $n\in\omega$, some $\kappa_n$ many of the components of $S_n$ are disjoint from both $\bigcup_{i<n}S_i$ and $U$.
Thus, the union of all stars $S_n$ contains $\kappa$ many internally disjoint $\eps$--$U$ combs.
Therefore, we may apply Lemma~\ref{lem_countablecore} to $\eps$ and $U$ to find the desired $\kappa$-star of rays. 
\end{proof}

\begin{theorem}
\label{thm_cof>omega1case} 
Under \weakch, $\HC(\kappa)$ holds for all $\kappa$ with  $\cf\p{\kappa} \notin \set{\mu^+}:{\cf\p{\mu} = \omega}$.
\end{theorem}

\begin{proof}
Let $\eps$ be an end of $G$ with $\deg(\eps) =  \kappa$. 
By the previous results, we may assume that $\kappa$ is singular and $\cf\p{\kappa} > \aleph_1$. 
Hence, by the Greedy Lemma~\ref{lem_greedy}, there are some set of vertices $U \subset V(G)$ with $|U| < \kappa$ and a $\kappa$-sized family $\mathcal{C}$ of internally disjoint $\eps$--$U$ combs. 
Consider the  $(\vert U\vert,\kappa)$-minor $H=(U,\mathcal{C})$ of $G$ where we contract the interior of every $\eps$--$U$ comb $C \in \mathcal{C}$. 
By Lemma~\ref{lem_final}, $H$~contains either an $(\aleph_0,\kappa)$-subgraph (in which case we are done by Lemma~\ref{lem_countablecore}), 
or a collection of disjoint $(\aleph_0,\kappa_i)$-subgraphs for $\set{\kappa_i}:{i < \cf\p{\kappa}}$ cofinal in $\kappa$ with all $\kappa_i > \max \Set{\aleph_{\omega+1}, \cf\p{\kappa}}$ regular.

Consider one $(\aleph_0, \kappa_i)$-subgraph $H_i$. As \weakch\ implies in particular that $2^{\aleph_0} < \aleph_{\omega+1}$, it follows from  $\aleph_{\omega+1} < \kappa_i$ and the regularity of $\kappa_i$ that there is a complete $(\aleph_0, \kappa_i)$-subgraph $H'_i\subset H_i$. 
Uncontracting the $\kappa_i$-side of $H'_i$ to combs and applying Lemma~\ref{lem_countablecore} inside the resulting subgraph of~$G$ (in which by construction all spines of the combs are still equivalent) gives a star of rays $S_i$ of size $\kappa_i$. By construction, any two such stars are disjoint.

Now, we apply Proposition~\ref{thm_aleph_n} to the collection $\mathcal{R}$ of all center rays of the $S_i$  to obtain a star of rays $S$ of size $\cf(\kappa)$ with leaf rays in $\mathcal{R}$. Keeping only those $S_i$ whose centre ray is a leaf ray of~$S$, we may assume that $S$ has precisely $\mathcal{R}$ as set of leaf rays. Since $|S| < |S_i|$  for all $i$, we may assume that each $S_i$ meets $S$ only in the former's centre ray. 
Then $S \cup \bigcup \set{S_i}:{i < \cf\p{\kappa}}$ yields a connected ray graph of size $\kappa$.
\end{proof}

\section{The first counterexample to Halin's conjecture}
\label{sec_counter}

\subsection{Order trees, \texorpdfstring{$T$}{T}-graphs and ray inflations} 
A partially ordered set $(T,\le)$ is called an \emph{order tree} if it has a unique minimal element (called the \emph{root}) and all subsets of the form $\lceil t \rceil = \lceil t \rceil_T := \set{t' \in T}:{t'\le t}$
are well-ordered. Write  $\lfloor t \rfloor = \lfloor t \rfloor_T  := \set{t' \in T}:{t\le t'}$.

 A~maximal chain in~$T$ is called a \emph{branch}
of~$T$; note that every branch inherits a well-ordering from~$T$. 
The \emph{height} of~$T$ is the supremum of the order types of its branches. 
The \emph{height} of a point $t\in T$ is the order type of~$\mathring{\lceil t \rceil}  :=
\lceil t \rceil  \setminus \{t\}$. The set $T^i$ of all points at height $i$ is
the $i$th \emph{level} of~$T$, and we
write $T^{<i} := \bigcup\set{T^j}:{j < i}$ as well as $T^{\leq i} := \bigcup\set{T^j}:{j \leq i}$.
 If $t < t'$, we use the usual interval notation  $(t,t') = \{s \colon t< s  < t'\}$ for nodes between  $t$ and $t'$.
If there is no point between $t$ and~$t'$, we call $t'$ a \emph{successor}
of~$t$ and $t$ the \emph{predecessor} of~$t'$; if $t$ is not a successor
of any point it is called a \emph{limit}.

An order tree $T$ is \emph{normal} in a graph $G$, if $V(G) = T$
and the two endvertices of any edge of $G$ are comparable in~$T$. We call $G$ a
\emph{$T$-graph} if $T$ is normal in $G$ and the set of lower neighbours of
any point $t$ is cofinal in $\mathring{\lceil t \rceil}$. 
For detailed information on normal tree orders, see also \cite{brochet1994normal}.

Given $T$ an order tree, a $T$-graph $G$ is \emph{sparse} if
the down-neighbourhood of any node $t$ is 
of order type $\cf\big(\mathring{\downcl{t}}\big)$. For trees $T$ of height at most $\omega_1$, this means that if $t$ is a successor, it has a unique down-neighbour namely its predecessor, and if $t$ is a limit, its down-neighbours form a cofinal $\omega$-sequence in $\mathring{\downcl{t}}$. (For $T$ an ordinal, this corresponds to a \emph{ladder system} in Todor\v{c}evi\'c's terminology \cite{todorcevic2007walks}).

An \emph{Aronszajn tree}  is an order tree of size $\aleph_1$ where all branches and levels are countable.

We  now introduce a concept that is crucial for all our counterexamples to Halin's conjecture.

\begin{defn}
\label{def_rayinflation}
Let $G$ be a sparse $T$-graph for an order tree $T$ of height at most $\omega_1$. The \emph{ray-inflation} $G \rayinf \N$ of $G$ is the graph with vertex set $T \times \N$, and the following edges (cf.~Figure~\ref{fig:RayInflation}): 
\begin{enumerate}
	\item For every $t \in T$ we add all the edges $(t,n)(t,n+1)$ with $n\in\N$ so that $R_t:=G[\,\{t\}\times\N\,]$ is a \emph{horizontal ray}.
	\item If $t \in T$ is a successor with predecessor $t'$, we add all edges $(t,n)(t',n)$ for all $n \in \N$. 
	\item If $t \in T$ is a limit with down-neighbours $t_0 <_T t_1 <_T t_2 <_T \ldots$ in $G$ we add the edges $(t,n)(t_n,n)$ for all $n \in \N$.
\end{enumerate}
\end{defn}
\begin{figure}[ht]
    \centering
\begin{tikzpicture}
\def\limity{6}

\draw[fill] (0,4.5) circle (.05);
\draw[->] (5,4.5) -- (10,4.5);

\foreach \y in {0,...,3} {
    \draw[fill] (0,\y+.5) circle (.05);
    \draw[->] (5,\y+.5) -- (10,\y+.5);
    \foreach \x in {5,...,9} {
        \draw[fill] (\x,\y+.5) circle (.05);
    }
}

\draw[->] (0,0) -- (0,5);
\draw[decorate sep={.2mm}{.8mm},fill] (0,5.33) -- (0,5.67);

\foreach \y in {4,...,0} {
    \draw[black] (0,\y) to[out=120,in=-120] (0,\limity);
    \draw[black] (5+\y,\y) to[out=120,in=-120] (5+\y,\limity);
    \draw[fill] (0,\y) circle (.05);
    \node at (0.4,\y) {\small $t_{\y}$};
    \draw[->] (5,\y) -- (10,\y);
    \node at (10.4,\y) {\small $R_{t_{\y}}$};
    \foreach \x in {5,...,9} {
        "\ifthenelse{\y=4}{\draw[decorate sep={.2mm}{.8mm},fill] (\x,\y+.14+.5) -- (\x,\y+1); \draw[fill] (\x,\y+.5) circle (.05); \draw (\x,\y) -- (\x,\y+.5);}{}"
        "\ifthenelse{\y>0}{\draw (\x,\y) -- (\x,\y-1);}{}"
        \pgfmathparse{\y+5}\edef\z{\pgfmathresult}
        "\ifthenelse{\x=\z}{\draw[fill, black] (\x,\y) circle (.05);\node at (\x+.45,\y-.2) {\footnotesize $(t_{\y},\y)$};}{\draw[fill] (\x,\y) circle (.05);}"
    }
}

\draw[fill] (0,\limity) circle (.05);
    \node at (0.4, \limity) {\small $t$};
\draw[->] (5,\limity) -- (10,\limity);
\node at (10.32,\limity) {\small $R_t$};
\foreach \x in {0,...,4} {
    \draw[fill] (\x+5,\limity) circle (.05);
    \node at (\x+5,\limity+.4) {\footnotesize $(t,\x)$};
}
\node at (0,-1) {$G$};
\node at (7.5,-1) {$G \rayinf \N$};
\end{tikzpicture}
\caption{The ray inflation of an $(\omega+1)$-graph.}
    \label{fig:RayInflation}
\end{figure}
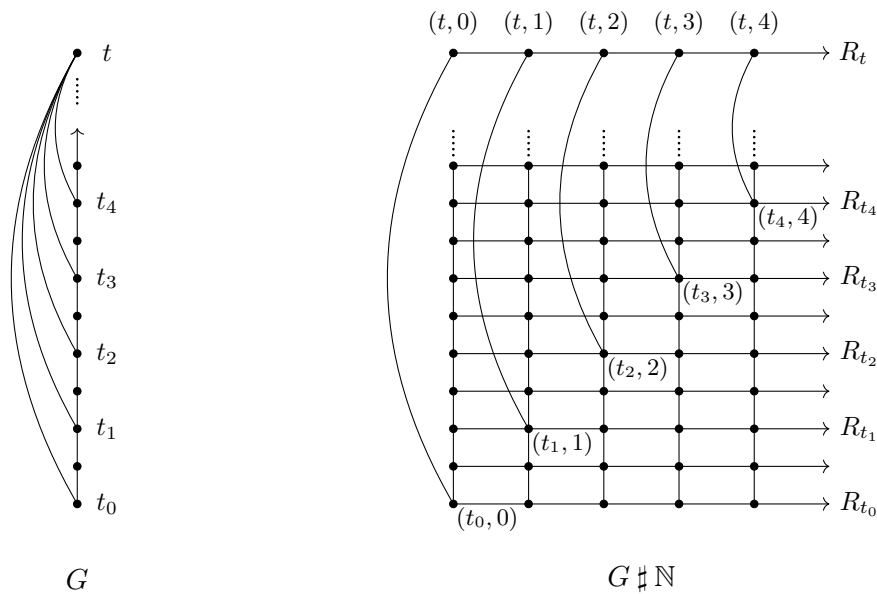

\begin{lemma}
\label{lem_ray_inflation_endstructure}
If $G$ is a sparse $T$-graph for $T$ an order tree of height at most $\omega_1$, then all the pairwise disjoint horizontal rays $R_t$ in the ray inflation $H=G \rayinf \N$ belong to the same sole end $\eps$ of $G \rayinf \N$; in particular, $\deg(\eps)=\vert T\vert$.
\end{lemma}

\begin{proof}
We show by an ordinal induction on $i\ge 0$ that the rays in the set $\script{R}^{\le i}:=\{R_t\colon t\in T^{\le i}\}$ belong to the same sole end of $H^{\le i}:=H[T^{\le i}\times\N]$.
For $i$ equal to the height of $T$, this implies the statement of the lemma.
The induction starts because $\script{R}^{\le 0}$ consists of the ray~$H^{\le 0}=R_r$ for $r$ the root of~$T$.
Now suppose that $i>0$.
We have to show that every horizontal ray $R_t$ with $t\in T^i$ is equivalent in $H^{\le i}$ to some other horizontal ray $R_{s}$ with $s\in T^{<i}$.
For this, let any horizontal ray~$R_t$ with $t\in T^i$ be given.
Since the vertex set of $H^{<i}$ is partitioned by $\script{R}^{<i}$ into vertex sets of rays that belong to the same sole end of $H^{<i}$ by the induction hypothesis, it suffices to show that $R_t$ can be extended to a comb in $H^{\le i}$ attached to $H^{<i}$.
If $t$ is a successor in $T$ with predecessor~$t'$, then such a comb arises from $R_t$ by adding the edges $(t,n)(t',n)$ for all $n\in\N$.
And if $t$ is a limit in~$T$ with down-neighbours $t_0<_T t_1<_T t_2<_T\ldots$ in $G$, then such a comb arises from $R_t$ by adding the edges $(t,n)(t_n,n)$ for all $n\in\N$.
\end{proof}

\begin{lemma}
\label{lem_ray_inflation_components}
Let $G$ be a sparse $T$-graph for $T$ an order tree of height at most $\omega_1$, let $0 \leq i < \omega_1 $ and let $H$ be the ray inflation of $G$.
Then the map 
$T^i\ni t\mapsto H[\,\upcl{t}\times\N\,]$ is a bijection between $T^i$ and the components of $H-(T^{<i}\times\N)$.
\end{lemma}
\begin{proof}
The $T$-graph $G$ is the contraction minor of $H$ that arises by contracting every horizontal ray $R_t$ to a single vertex.
For $T$-graphs such as $G$ it is well known, and straightforward to show, that the map $T^i\ni t\mapsto G[\,\upcl{t}\,]$ is a bijection between $T^i$ and the components of $G-T^{<i}$.
Since every component of $H-(T^{<i}\times\N)$ arises from a component of $G-T^{<i}$ by uncontracting every vertex of that component to a horizontal ray, the claim follows.
\end{proof}

\subsection{An Aronszajn tree of rays}

Our  counterexample to $\HC(\aleph_1)$ is based on the ray inflation of an Aronszajn tree.
Diestel, Leader and Todor\v{c}evi\'c showed in~\cite{DiestelLeaderNST} that:

\begin{prop}
\label{thm_dl_Aronszajn_with_property}
There exist an Aronszajn tree $T$ and a sparse $T$-graph $G$ with the following property:
\begin{fleqn}%
\begin{equation*}%
\tag{$\star$}%
\hspace{2\parindent}\begin{aligned}\label{dlp}%
    \parbox{\textwidth-5\parindent}{For every $t\in T$ there is a finite $S_t\subset\mathring{\downcl{t}}$ such that every $t'>t$ has all its\\down-neighbours below $t$ inside $S_t$.}%
\end{aligned}%
\end{equation*}%
\end{fleqn}%
\end{prop}

\begin{proof} We sketch the construction from \cite[Theorem~6.2]{DiestelLeaderNST} for convenience of the reader.
Let $T$ be an Aronszajn tree with an antichain partition $\set{U_n}:{n \in \N}$ (the standard Aronszajn tree constructions yield such an antichain partition).

Given a limit $t \in T$, choose its down-neighbours $t_0 < t_1 < t_2 < \ldots$ inductively, starting with~$t_0$ the root of $T$. If $t_{n-1} < t$ has already been defined, consider the least $i  \in \N$ such that the antichain~$U_i$ meets the interval $(t_{n-1},t)$, and let $t_n$ be the (unique) point in $U_i \cap (t_{n-1},t)$. It is easy to check that the $T$-graph $G$  is as desired.
\end{proof}

\begin{lemma}
\label{lem_dlp_to_dlmp}
If $T$ is an Aronszajn tree and $G$ is any $T$-graph with property~\emph{(\ref{dlp})}, then the ray inflation $H=G \rayinf \N$ of $G$ has the following property:
\begin{fleqn}%
\begin{equation*}%
\tag{$\star\star$}%
\hspace{2\parindent}\begin{aligned}\label{dlmp}%
    \parbox{\textwidth-5\parindent}{For every $t\in T$ there is a finite $S_t\subset\mathring{\downcl{t}}$ such that every $(t',n)\in H$ with $t'>_T t$\\and $n\in\N$ satisfies $N_H(\,(t',n)\,)\cap (\mathring{\downcl{t}}\times\N)\subset S_t\times\vert S_t\vert$.}%
\end{aligned}%
\end{equation*}%
\end{fleqn}%
\end{lemma}
\begin{proof}
Given $t\in T$, and given $S_t$ by property~(\ref{dlp}), we show that every $(t',n)\in H$ with $t'>_T t$ and $n\in\N$ satisfies $N:=N_H(\,(t',n)\,)\cap (\mathring{\downcl{t}}\times\N)\subset S_t\times\vert S_t\vert$.
If $t'$ is a successor, then $N$ is empty, so $N\subset S_t\times\vert S_t\vert$.
Otherwise $t'$ is a limit and has down-neighbours $t'_0<_T t'_1<_T t'_2\ldots$ in $G$.
Then $\{t'_n\colon n\in\N\}\cap \mathring{\downcl{t}}=\{t'_0,\ldots,t'_k\}\subset S_t$ for $k<\vert S_t\vert$, so $N=\{\,(t'_0,0),\ldots,(t_k',k)\,\}\subset S_t\times\vert S_t\vert$.
\end{proof}

\begin{theorem}
\label{thm_ATcounterexample}
Let $T$ be an Aronszajn tree and let $G$ be a $T$-graph with property~\emph{(\ref{dlp})} as in Proposition~\ref{thm_dl_Aronszajn_with_property}.
Then the ray inflation $G \rayinf \N$  of  $G$ witnesses that $\HC(\aleph_1)$ fails.
\end{theorem}

\begin{proof}
Suppose for a contradiction that for every end $\eps$ of $H=G \rayinf \N$ there is a set $\script{R}\subset\eps$ of disjoint rays with $\vert\script{R}\vert=\deg\p{\eps}$ such that some ray graph of $\script{R}$ in $G$ is connected.
By Lemma~\ref{lem_ray_inflation_endstructure}, all the horizontal rays $R_t$ ($t\in T$) of $G \rayinf \N$ belong to the same sole end $\eps$ with $\deg\p{\eps}=\vert T\vert=\aleph_1$.
Hence, by our assumption and Lemma~\ref{lem_fundamentalRayGraphs} we find an $\aleph_1$-star of rays in~$G \rayinf \N$ which we denote by~$S$.
We denote the centre ray of $S$ by $R$, and we denote the leaf rays of $S$ by $R_i$ ($i<\omega_1$).

Now, let $\sigma<\omega_1$ be minimal such that $V(R)\subset T^{<\sigma}\times\N$.
Since $T^{<\sigma}\times\N$ is countable, we may assume without loss of generality that $S$ meets $T^{<\sigma}\times\N$ precisely in~$R$.
Then every component of $S-R$ is contained in a component of $H-(T^{<\sigma}\times\N)$.
By Lemma~\ref{lem_ray_inflation_components}, the map $T^\sigma\ni t\mapsto H[\,\upcl{t}\times\N\,]$ is a bijection between $T^\sigma$ and the components of $H-(T^{<\sigma}\times\N)$.
Thus, we obtain a map $\omega_1\to T^\sigma$, $i\mapsto t_i$ such that the component $D_i$ of $S-R$ with $D_i\supset R_i$ is contained in the component $C_i:=H[\,\upcl{t_i}\times\N\,]$ of $H-(T^{<\sigma}\times\N)$.
Since the level $T^\sigma$ of the Aronszajn tree $T$ is countable, some $t\in T^\sigma$ is the image $t_i$ of uncountably many indices $i<\omega_1$, and we abbreviate $C_i=:C$ for these indices.
Then for uncountably many indices $i<\omega_1$ of these uncountably many indices their components $D_i$ are contained in $C-(\{t\}\times\N)$ entirely.
But then all these components $D_i$ have infinite neighbourhood in $\mathring{\downcl{t}}\times\N$, contradicting that $H$ has property~(\ref{dlmp}) by Lemma~\ref{lem_dlp_to_dlmp}.
\end{proof}

We remark that  the ray inflation $G \rayinf \N$ of any (sparse) $\omega_1$-graph $G$ does contain an $\aleph_1$-star of rays.  
Indeed, by \cite[Proposition~3.5]{DiestelLeaderNST} applied to the minor $G$ of $G \rayinf \N$, the ray inflation $G \rayinf \N$ contains a $K^{\aleph_1}$ minor, and so in fact a subdivision of $K^{\aleph_1}$ by a well-known result of Jung~\cite{jung1967zusammenzuge}.
In particular, the ray inflations giving counterexamples to $\HC(\aleph_1)$ must be based on Aronszajn trees.

We remark that, based on this example, one can  construct counterexamples to Halin's conjecture for all $\kappa$ with $\cf\p{\kappa} = \aleph_1$,  see Theorem~\ref{thm_cofinalityupwards}.

\section{On \texorpdfstring{$(\lambda,\kappa)$}{(lambda,kappa)}-graphs II -- regular trees with  tops}
\label{sec_lambdakappaII}

Earlier, in Section~\ref{sec_lambdakappa} we investigated conditions under which $(\lambda,\kappa)$-graphs possess $(\aleph_0,\kappa)$-subgraphs, 
and used these results to prove a number of positive instances of Halin's conjecture in Section~\ref{sec_affirm}.
However, there exist $(\lambda,\kappa)$-graphs without $(\aleph_0,\kappa)$-subgraphs, and the question arises whether these can be turned into counterexamples for $\HC(\kappa)$.
And indeed, our main result in this section, Theorem~\ref{thm_MaxRewriteStefan}, states that there is such a class of $(\lambda,\kappa)$-graphs, 
which we call $(\lambda,\kappa)$-graphs of type $T_\lambda$, that achieve precisely this,
see Theorem~\ref{thm_scalecounterexamples} below.

Given a cardinal $\lambda \ge 2$, either finite or of countable cofinality, let $(T_\lambda,\leq)$ be the order tree where the nodes of $T_\lambda$ are all sequences of elements of $\lambda$ of length $\leq \omega$ including the empty sequence, and $t \leq t'$ if $t$ is an initial segment of $t'$. Then $T_\lambda$ is an order tree of height $\omega+1$ in which every point of finite height has exactly $\lambda$ successors and above every branch of $T^{<\omega}_
\lambda$ there
is exactly one point in $T^\omega_\lambda$, represented by a countable sequence of ordinals in $\lambda$. 
Since $\lambda \ge 2$ is finite or of countable cofinality, it follows from K\"onig's Theorem \cite[Theorem~5.10]{jech2013set} that $T_\lambda$ has strictly more  than $\lambda$ many branches, that is to say we have $|T^\omega_\lambda| \geq \max \{\lambda^+,2^{\aleph_0}\}$.

The down-closure of any $\kappa$-sized
subset $X$ of $T^\omega_\lambda$ for $\kappa > \lambda$
is a \emph{$\lambda$-tree with $\kappa$ tops}, the tops themselves being the points
in $X$. 
Now if $T$ is a $\lambda$-tree with $\kappa$ tops, then any $T$-graph clearly contains
a $(\lambda,\kappa)$-graph with bipartition classes $T^{<\omega}$ and $T^{\omega}$; we shall call any such graph a $(\lambda,\kappa)$-graph of \emph{type $T_\lambda$}.

We remark that $\lambda$-trees with $\kappa$ tops form a generalisation of the so-called \emph{binary trees with tops}, studied in more detail in \cite{BowlerGeschkePitzNST,DiestelLeaderNST}, where it  was shown that, consistently, every $(\aleph_0,\aleph_1)$-graph contains an
$(\aleph_0,\aleph_1)$-subgraph of type $T_2$  \cite[Theorem~1.1]{BowlerGeschkePitzNST}, a statement which does not hold under \CH\ \cite[Proposition~8.2]{DiestelLeaderNST}.

Our main result in this section, which forms the basis for \ref{item_ZFCcounterexample} in Theorem~\ref{thm_mainresult}, is the following:

\begin{theorem}\label{thm_MaxRewriteStefan}
For any singular cardinal $\lambda$ of countable cofinality there is a  $(\lambda,\lambda^+)$-graph of type~$T_\lambda$ that does not have an $(\aleph_0,\lambda^+)$-subgraph.
\end{theorem}

Before we proceed to the proof, we also state two consistent  results about the number of branches of a certain pair of such trees. 
For this, recall that \GCH\ implies that every $\lambda$-tree $T_\lambda$ from above for $\lambda$ of countable cofinality has precisely $\lambda^+$ branches \cite[Theorem~5.15(ii)]{jech2013set}, in which case every $(\lambda,\kappa)$-graph of type $T_\lambda$ satisfies $\kappa  =  \lambda^+$. 

Of course, one way to increase the number of branches of $T_\lambda$ is to work in a model of \ZFC\ where $2^{\aleph_0}$ is large.

It turns out, however, that for at least two trees, namely $T_{\aleph_\omega}$ and $T_\mu$ where $\mu = \aleph_\mu$ denotes the least $\aleph$ fixed point (which is a singular cardinal of countably cofinality, namely $\mu  = \sup \{\aleph_0,\aleph_{\aleph_0},\aleph_{\aleph_{\aleph_0}},\ldots \}$), it is known that the number of branches can consistently be much larger despite the continuum being small, i.e.\  $2^{\aleph_0}=\aleph_1$. This gives rise to  our second result in this section, forming the basis for \ref{item_consistentcounterexamples} in Theorem~\ref{thm_mainresult}:

\begin{theorem}
\label{thm_MaxRewriteStefan2}
\begin{enumerate}
    \item For any countable ordinal $\alpha$ with $\omega < \alpha < \omega_1$, it is consistent that we have  \CH\ and there is an  $(\aleph_\omega,\aleph_{\alpha})$-graph of type $T_{\aleph_\omega}$ that does not have an $(\aleph_0,\aleph_{\alpha})$-subgraph.
    \item Let $\mu=\aleph_\mu$ denote the first fixed point of the $\aleph$-function. 
 Then for every $\kappa>\mu$ it is consistent that we have  \CH\ and there is a
 $(\mu,\kappa)$-graph of type $T_\mu$ that contains no  $(\aleph_0,\kappa)$-subgraph.
\end{enumerate}
\end{theorem}

Given the 
result that the trees $T_{\aleph_\omega}$ and $T_\mu$ can consistently have as many branches as needed, the proof  of Theorem~\ref{thm_MaxRewriteStefan2} is remarkably simple.

\begin{proof}
(1) Assuming large cardinals, Gitik and Magidor \cite{GitikMagidor} showed that $T_{\aleph_\omega}$ can consistently have any number of branches of the form $\aleph_{\omega+\alpha+1}$ where $\alpha<\omega_1$, while simultaneously having \GCH\ below $\aleph_\omega$ (i.e.\ $2^{\aleph_n} = \aleph_{n+1}$ for all $n<\omega$). 
Now let $\omega < \alpha < \omega_1$ and consider any $(\aleph_\omega,\aleph_{\alpha})$-graph of type $T_{\aleph_\omega}$ in such a model, and suppose for a contradiction that it contains an $(\aleph_0,\aleph_{\alpha})$-subgraph $(A,B)$. Without loss of generality, $B$ is a set of tops of $T_{\aleph_\omega}$, and $A \subset T_{\aleph_\omega}$ is a down-closed subtree. But a countable tree contains at most $2^{\aleph_0}  =\aleph_1$ many branches, contradicting that every top in $ B$ has infinitely many neighbours in $A$.

(2) Assuming even larger cardinals, Shelah showed that $T_{\mu}$ can consistently have arbitrarily many branches, while simultaneously having \GCH\ below $\mu$. See again~\cite{GitikMagidor}. Now given $\kappa > \mu$, consider any $(\mu,\kappa)$-graph of type $T_{\mu}$ in such a suitable model in which $|T_{\mu}^\omega| \geq \kappa$ is sufficiently large. As before, such a graph cannot contain an $(\aleph_0,\kappa)$-subgraph.
\end{proof}

Note that by a similar argument, one can obtain a simple proof of Theorem~\ref{thm_MaxRewriteStefan} under \GCH. Perhaps remarkably, however, this  assertion holds already in ZFC. 
Our examples rely on the notion of a \emph{scale} (see also \cite[Chapter~24]{jech2013set}) that have been developed for Shelah's pcf-theory \cite{CardinalArithmetic}. 
Recall that an \emph{ideal} on the natural numbers $\N$ is a proper subset of~$\mathcal{P}(\N)$ that contains the empty set, is closed under finite unions and is closed under taking subsets of its elements.
Thus, it is the dual notion of a filter~\cite[Chapter~7]{jech2013set}.

Given an ideal $I$ on $\N$ and two sequences $f,g \colon \N \to \lambda$ of ordinals, we write 
$f<_Ig$ if 
$$\{n\in\N\colon f(n)\geq g(n)\}\in I.$$

\begin{defn}[Scales]
Let $\lambda$ be a singular cardinal of countable cofinality and $\kappa > \lambda$ regular. 
A~\emph{$\kappa$-scale} for $T_\lambda$ is a well-ordered collection $X =(f_\alpha)_{\alpha<\kappa}$ of tops of $T_\lambda$ for which there are 
\begin{itemize} 
    \item a strictly increasing sequence $(\lambda_n)_{n\in\N}$ of uncountable regular cardinals  with supremum $\lambda$ satisfying $f_\alpha(n) < \lambda_n$ for all $\alpha < \kappa$, and
    \item an ideal $I$ on $\N$ containing all finite sets, 
\end{itemize}
such that 
\begin{enumerate}
    \item for all $\alpha,\beta<\kappa$ with $\alpha<\beta$ we have $f_\alpha<_If_\beta$ and
    \item for all $g\in\prod_{n\in\N}\lambda_n$ there is $\alpha<\kappa$ such that $g<_If_\alpha$.
\end{enumerate}
\end{defn}

\begin{prop}\label{ExamplesFromScales}
Let $\lambda$ be a singular cardinal of countable cofinality. Given any $\kappa$-scale $X$, the corresponding tree $T_\lambda$ with tops $X$ gives rise to a $(\lambda,\kappa)$-graph that has no $(\aleph_0,\kappa)$-subgraph.
\end{prop}

\begin{proof}
Let $X=(f_\alpha)_{\alpha<\kappa}$ be a $\kappa$-scale on $T_\lambda$. Let $T$ be the corresponding $\lambda$-tree with tops $X$.

Suppose $S\subseteq T$ is a countable subtree of $T$.
Then there is a function
$g \in \prod_{n\in\N}\lambda_n$ such that for all
$t\in S$ and all $n$ in the domain of $t$ we have
$t(n)<g(n)$.
Let $\alpha<\kappa$ be such that $g<_If_\alpha$.
Let $\beta<\kappa$ be such that $\alpha<\beta$.
Then $g<_If_\beta$.
Consider the set
$$A=\{n\in\N\colon g(n)\geq f_\beta(n)\}\in I.$$
Since $I$ is an ideal, we have $A\subsetneq \N$.
Hence, for some $n\in\N$ and all $t\in S$
with $n$ in the domain of~$t$,
$f_\beta(n)>t(n)$.
It follows that $f_\beta$ is not a branch of $S$.
This shows that $S$ has at most $|\alpha|<\kappa$
branches in $X$.

Now let $G'$ be any $T$-graph and $G \subset G'$ the corresponding $(\lambda,\kappa)$-graph with bipartition classes $T^{<\omega}$ and $X$. 
If $H$ is an $(\aleph_0,\kappa)$-subgraph $(C,D)$ of $G$, then without loss of generality $D \subset X$ and $C \subset  T^{<\omega}$, and we can choose a countable subtree $S$ of $T$ such that $C\subseteq S$.
Let $g$ and $\alpha$ be as in the argument above for the countable tree $S$.

Since $H$ is an $(\aleph_0,\kappa)$-graph, there is $\beta<\kappa$ such that $f_\beta\in D$ and $\beta>\alpha$.
As before, for some $n\in\N$ and all $t\in S$,
whenever $n$ is in the domain of $t$, then $f_\beta(n)>t(n)$.  It follows that for no $m\in\N$ with $m>n$, we have $f_\beta\restriction m\in S$.  Hence $f_\beta$ only has finitely many neighbours in $C$, contradicting the assumption that $H$ is an $(\aleph_0,\kappa)$-graph.
\end{proof}

\begin{proof}[Proof of Theorem~\ref{thm_MaxRewriteStefan}]
The assertion follows from Proposition~\ref{ExamplesFromScales} together with the well-known result by Shelah that for any singular cardinal $\lambda$ of countable cofinality, there is a $\lambda^+$-scale on $T_\lambda$, see \cite[Theorem~24.8]{jech2013set}.
\end{proof}

\section{More counterexamples to Halin's conjecture}
\label{sec_counterII}

We are now ready to turn the $(\lambda,\kappa)$-graphs of type $T_\lambda$ from the previous section into counterexamples for $\HC(\kappa)$.

\begin{theorem}
\label{thm_scalecounterexamples}
Let $T$ be a $\lambda$-tree with $\kappa$ tops $X$, and $G$ any sparse $T$-graph such that the corresponding  $(\lambda,\kappa)$-graph 
on $(T^{<\omega},X)$ has  no $(\aleph_0,\kappa)$-subgraph. Then the ray inflation $G \rayinf \N$  of  $G$ witnesses that $\HC(\kappa)$  fails.
\end{theorem}

\begin{proof}
Suppose for a contradiction that $\HC(\kappa)$ holds.
By Lemma~\ref{lem_ray_inflation_endstructure}, the ray inflation $H=G \rayinf \N$ has only one end $\eps$, and $\deg (\eps)=\vert T\vert=\kappa$.
Then by Lemma~\ref{lem_fundamentalRayGraphs} we find a $\kappa'$-star $S$ of rays in~$H$  with $\lambda< \kappa' \leq \kappa$ where $\kappa'$ is regular. For ease of notation, let us assume $\kappa = \kappa'$.
Denote the centre ray of $S$ by~$R$, and the leaf rays of $S$ by $R_i$ ($i<\kappa$).

Since $T^{<\omega}\times\N$ has size $\lambda<\kappa$, we may assume without loss of generality that each leaf ray $R_i$ is a tail of a horizontal ray $R_{t(i)}\subset H$ for a top $t(i)\in T^\omega$.
Then the map $i\mapsto t(i)$ is injective.
Next, we consider the `down-closed projection' of the center ray $R$ to $T$, namely \[\bar{R}:=\downcl{\{\,t\in T\mid \text{the horizontal ray }R_t\text{ meets }R\,\}},\]
a countable subtree of $T$.
We claim that the $\lambda$-set $X:=(T^{<\omega}\setminus\bar{R})\times\N$ contains $\kappa>\lambda$ many internal vertices of paths in the path system of~$S$, causing a contradiction.
By the choice of $T$, fewer than $\kappa$ many tops $t\in T^\omega$ satisfy $N_G(t)\subset \bar{R}$.
Hence we may assume without loss of generality that each leaf ray $R_i$ is a tail of a horizontal ray $R_{t(i)}$ with $N_G(t(i))\not\subset\bar{R}$ and, in particular, $N_G(t(i))\cap\bar{R}$ finite.
But then for every $i<\kappa$ all but finitely many vertices of the neighbourhood $N_H(R_{t(i)})$ are contained in~$X$.
Thus, all paths of the path system of~$S$ must have an internal vertex in~$X$, as desired.
\end{proof}

\begin{cor}\label{cor_theexamples}\leavevmode
\begin{enumerate}
 \item $\HC(\kappa)$ fails for all $\kappa \in \set{\mu^+}:{\cf\p{\mu} = \omega}$.
 \item  For every $\kappa \in \set{\aleph_{\alpha}}:{\omega<\alpha<\omega_1}$ it  is consistent that $\HC(\kappa)$ fails.
 \item Let $\mu$ denote the first fixed point of the $\aleph$-function. 
 Then for every  $\kappa>\mu$ it is consistent that $\HC(\kappa)$  fails.
\end{enumerate}
\end{cor}

\begin{proof}
Assertion (1) for  $\kappa  = \aleph_1$ is Theorem~\ref{thm_ATcounterexample}. For the remaining cardinals in (1), it follows from Theorem~\ref{thm_scalecounterexamples} together with Theorem~\ref{thm_MaxRewriteStefan}. Assertions (2) and (3) follow from Theorem~\ref{thm_scalecounterexamples} in combination with Theorem~\ref{thm_MaxRewriteStefan2}.
\end{proof}

\section{Lifting counterexamples to singular cardinals}
\label{sec_lifting}

\begin{theorem}
\label{thm_cofinalityupwards}
If $\HC(\kappa)$ fails for some regular cardinal $\kappa$, then $\HC(\lambda)$ fails for all cardinals $\lambda$ with $\cf(\lambda)=\kappa$.
\end{theorem}

Our proof strategy is roughly as follows.
Given $\lambda>\kappa$ we consider any graph $G$ with an end~$\eps$ witnessing that $\HC(\kappa)$ fails.
Then we obtain a counterexample $\hat{G}$ for $\HC(\lambda)$ from $G$ as follows.
We select any $\kappa$ many disjoint rays $X_i$ ($i<\kappa$) in $\eps$ and consider any sequence $s=(\lambda_i\colon i<\kappa)$ of ordinals $\lambda_i<\lambda$ with supremum~$\lambda$.
Then we obtain $\hat{G}$ from $G$ by adding $\kappa$ many disjoint $\lambda_i$-stars of rays all meeting $G$ precisely in their centre ray~$X_i$.
Next, in order to verify that $\hat{G}$ is a counterexample, we assume for a contradiction that $\HC(\lambda)$ holds.
Using $\HC(\lambda)$ and Lemma~\ref{lem_fundamentalRayGraphs} in~$\hat{G}$ we find either a $\lambda$-star of rays or a $(\lambda,s)$-star of rays, with all rays belonging to the end~$\hat{\eps}$ that includes~$\eps$.
A short argument shows that we cannot get a $\lambda$-star of rays.
Our aim then is to use the $(\lambda,s)$-star of rays in~$\hat{G}$ to find a $\kappa$-star of rays in~$G$ with all rays belonging to~$\eps$.
An obvious candidate is the $\kappa$-star of rays formed by the centre ray and the distributor rays of the $(\lambda,s)$-star of rays.
However, the
\begin{enumerate}
    \item distributor rays,
    \item paths from the distributor rays to the centre ray, and
    \item the centre ray
\end{enumerate}
need not be included in~$G$. In the remainder of the proof, we adjust our candidate in three steps $i=1,2,3$ so that ($i$) is included in~$G$ at the end of step~$i$.

\begin{proof}[Proof of Theorem~\ref{thm_cofinalityupwards}]
Suppose that $\HC(\kappa)$ fails for some regular cardinal~$\kappa$ and let $\lambda>\kappa$ be any other cardinal with $\cf(\lambda)=\kappa$.
Then $\kappa>\aleph_0$ and there are a graph~$G$ with $|G|= \kappa$ and an end $\eps$ of~$G$ of degree~$\kappa$ such that no degree-witnessing collection of disjoint rays in $\eps$ admits a connected ray graph.
Let $(X_i\colon i<\kappa)$ be any family of $\kappa$ disjoint rays in~$\eps$, and let $s=(\lambda_i\colon i<\kappa)$ be any $\kappa$-sequence of ordinals $\lambda_i<\lambda$ with supremum~$\lambda$.
Then we let $\hat{G}$ be the graph obtained from $G$ by adjoining $\lambda$ many new rays, as follows.
For each $i<\kappa$ and $\ell<\lambda_i$ we disjointly add a new ray $X_{(i,\ell)}$ and join the $n$th vertex of $X_{(i,\ell)}$ to the $n$th vertex of $X_i$ by an edge for all~$n\in\N$.
Then the end $\eps$ is included in a unique end $\hat{\eps}$ of $\hat{G}$, and $\hat{\eps}$ contains all new rays~$X_{(i,\ell)}$ so that it has degree~$\lambda$.
For each $i<\kappa$ we let $\hat{X}_i:=\hat{G}[X_i\cup\bigcup_{\ell<\lambda_i}X_{(i,\ell)}]$ be the $\lambda_i$-star of rays $X_i$ and~$X_{(i,\ell)}$.
We claim that $\hat{G}$ and $\hat{\eps}$ witness that $\HC(\lambda)$ fails.

Assume for a contradiction that $\HC(\lambda)$ holds.
Then, by Lemma~\ref{lem_fundamentalRayGraphs}, we find either a $\lambda$-star of rays in~$\hat{G}$ or a $(\lambda,s)$-star of rays in~$\hat{G}$, with all rays belonging to~$\hat{\eps}$.
Let $S\subset \hat{G}$ be one of the two possibilities with path system~$\script{P}$.
We denote the centre ray of $S$ by $R$, and we denote the leaf rays by $R_j$ ($j<\lambda$).
Since $|G| = \kappa<\lambda$, we may assume without loss of generality that every leaf ray~$R_j$ is a tail of a ray $X_{(i,\ell)}$.
Since $R$ is countable while $\cf(\lambda)=\kappa$ is uncountable, we may assume without loss of generality that every $\hat{X}_i$ that meets~$R$ avoids all the leaf rays of~$S$.
Furthermore, we may assume without loss of generality that each $\hat{X}_i$ either avoids all the leaf rays of~$S$ or includes uncountably many of them.
In summary, for all $i<\kappa$ the $\lambda_i$-star $\hat{X}_i$ of rays either avoids all the leaf rays of~$S$, 
or avoids $R$ while uncountably many leaf rays of~$S$ are tails of the rays~$X_{(i,\ell)}$.
Since every ray $X_i$ is countable, this means that $S$ cannot be a $\lambda$-star of rays.

Now that we know that $S$ must be a $(\lambda,s)$-star of rays, we denote its distributor rays by $R_i$ ($i<\kappa$), and we revise our notation for its leaf rays which we now denote by $R_{(i,\ell)}$ ($i<\kappa$ and $\ell<\lambda_i$) so that $R_i$ is neighboured precisely to the leaf rays $R_{(i,\ell)}$ with $\ell<\lambda_i$ and the centre ray~$R$.
Our goal is to find a $\kappa$-star of rays in $G$ with all rays belonging to~$\eps$.
For this, we consider the subset $I\subset\kappa$ of all $i<\kappa$ for which $\hat{X}_i$ avoids $R$ and uncountably many leaf rays of $S$ are tails of the rays $X_{(i,\ell)}$. Note that $I$ is a $\kappa$-set.
For each $i\in I$, the uncountably many relevant leaf rays are neighboured to at most countably many distributor rays: otherwise there exist a leaf ray $R_{(i',\ell')}\subset X_{(i,\ell)}$ and a distributor ray $R_j\subset \hat{G}-X_i$, implying $R_j\subset\hat{G}-X_i-X_{(i,\ell)}$, such that the infinitely many $R_{(i',\ell')}$--$R_j$ paths in~$\script{P}$ avoid~$X_i$, contradicting the fact that $X_i$ is equal to the neighbourhood of $X_{(i,\ell)}$ in~$\hat{G}$.
Thus, we may apply the pigeonhole principle to find, for each $i\in I$, a distributor ray $R_j$ with $j=:j(i)$ that is neighboured to uncountably many leaf rays of~$S$ which are tails of rays~$X_{(i,\ell)}$ ($\ell<\lambda_i$).
Since each $X_i$ is countable and $R_{j(i)}$ is the centre ray of an $\aleph_1$-star of rays with leaf rays in $\hat{X}_i-X_i$, we deduce that each $R_{j(i)}$ meets $X_i$ infinitely.

By deleting combs from $S$ and deleting paths from $\script{P}$ we may assume without loss of generality that $S$ is a $\kappa$-star of rays with centre ray $R$, leaf rays $R_{j(i)}$ ($i\in I$) and path system~$\script{P}$.

Every component of $S-R$ is a comb whose spine meets $G$ infinitely.
Since these combs are disjoint for distinct elements of~$I$ and $\hat{X}_i\cap G=X_i$ is countable for all $i<\kappa$, each $\hat{X}_i$ ($i<\kappa$) meets at most countably many of them.
Conversely, each comb meets at most countably many $\hat{X}_i$.
Therefore, by deleting elements of~$I$ we achieve that
\begin{equation}
\begin{aligned}\label{eq:IprimePropertyOne}
    &\textit{each $\hat{X}_i$ $(i<\kappa)$ meets at most one of $R$ and the components of $S-R$}
\end{aligned}
\end{equation}
while maintaining that $I$ is a $\kappa$-set.

For every $i\in I$ we let $\script{P}_i$ be the set of all $R_{j(i)}$--$R$ paths in~$\script{P}$.
As each $R_{j(i)}$ meets $X_i$ infinitely, $X_i$ is equivalent to $R$ in the subgraph $X_i\cup R_{j(i)}\cup\bigcup\script{P}_i\cup R$, and we find a system $\script{P}_i'$ of infinitely many disjoint $X_i$--$R$ paths in this subgraph.
Then the paths in $\script{P}':=\bigcup_{i\in I}\script{P}_i'$ are independent by the choice of the subgraphs in which we found them and because each $X_i$ meets $S$ only in the component of $S-R$ that contains~$R_{j(i)}$ by~(\ref{eq:IprimePropertyOne}).
Hence, $\script{P}'$ ensures that $R$ is the centre ray of a $\kappa$-star $S'$ of rays with leaf rays $X_i\subset G$ ($i\in I$).
Then (\ref{eq:IprimePropertyOne}) translates to
\begin{align}\label{eq:IprimePropertyTwo}
\textit{each $\hat{X}_i$ $(i<\kappa)$ meets at most one of $R$ and the components of $S'-R$.}
\end{align}

It remains to modify $S'$ so that $S'\subset G$.
By~(\ref{eq:IprimePropertyTwo}), the neighbourhood of each component of $S'-R$ in $S'$ is included in~$R\cap G$, so every path in $\script{P}'$ ends in $R$ in a vertex of~$G$.

If $P\subset\hat{G}$ is a path with endvertices in $G$, then we write $\pi(P)$ for the subgraph of $G$ that arises from $P$ by replacing each $G$-path $Q\subset P$ with the subpath $uX_iv$ between the endvertices $u$ and $v$ of $Q$ on the ray $X_i$ that contains $u$ and~$v$.

For every $i\in I$, each path $P'\in\script{P}_i'$ starts in a vertex $x\in X_i\subset G$ and ends in a vertex $y\in R\cap G$ while avoiding $\hat{X}_i-X_i$ and all $\hat{X}_j-X_j$ ($j<\kappa$) that are met by~$R$, so $\pi(P')$ contains an $X_i$--$R$ path $P''$ that starts in $x$ and ends in~$y$.
We claim that, for every $i\in I$ and every path $P_1'\in\script{P}_i'$, the subgraph $\pi(P'_1)$ meets only finitely many other subgraphs $\pi(P_2')$ with $P_2'\in\script{P}_i'$.
Indeed, given $i\in I$ let us assume for a contradiction that there are infinitely many paths $P'$ and $P_0',P_1',\ldots$ in $\script{P}_i'$ such that every subgraph $\pi(P_n')$ meets~$\pi(P')$.
Then there is a vertex $v\in \pi(P')$ that lies on infinitely many of the other subgraphs, say all of them.
Clearly, $v$ must lie on some ray~$X_j$ with $j\neq i$.
Since the paths $P_n'$ are disjoint, at most one can contain $v$, say none.
Then each path $P_n'$ contains a $G$-path $Q_n$ with endvertices $x_n$ and $y_n$ in $X_j$ such that $v\in\mathring{x}_nX_j\mathring{y}_n$.
But the initial segment $X_j v$ of the ray $X_j$ up to~$v$ is finite and cannot contain the infinitely many distinct vertices~$x_n$, a~contradiction.
Therefore, we find for every $i<\kappa$ an infinite subset $\script{P}_i''\subset\{P''\colon P'\in\script{P}_i'\}$ of disjoint $X_i$--$R$ paths all contained in~$G$.
By~(\ref{eq:IprimePropertyTwo}) and the choice of~$\script{P}'$,
the paths in $\script{P}'':=\bigcup_{i\in I}\script{P}_i''$ are independent.
Hence $\script{P}''$ ensures that $R$ is the centre ray of a $\kappa$-star $S''$ of rays with leaf rays $X_i$ ($i\in I$) all belonging to~$\eps$ and $S''-R\subset G$.

Now, $S''-R$ is included in~$G$, but $R$ need not be included in~$G$.
However, the endvertices in $R$ of paths in $\script{P}''$ all lie in~$G$, because they inherit this property from~$\script{P}'$.
Hence, applying  Lemma~\ref{lem_countablecore} with $\mathcal{C}$ the collection of combs $X_i\cup\bigcup\script{P}_i''$ ($i\in I$) and the countable set $U: = V(R) \cap V(G)$ in $G$ finishes the proof.
\end{proof}

\begin{cor}\label{cor_theexamples2}\leavevmode
\begin{enumerate}
    \item $\HC(\kappa)$ fails for all $\kappa$ with $\cf\p{\kappa} \in \set{\mu^+}:{\cf\p{\mu} = \omega}$.
 \item  For every $\kappa$ with $\cf\p{\kappa} \in \set{\aleph_{\alpha}}:{\omega<\alpha<\omega_1}$ it  is consistent that $\HC(\kappa)$ fails.
\end{enumerate}
\end{cor}

\begin{proof}
This follows from Theorem~\ref{thm_cofinalityupwards} together with Corollary~\ref{cor_theexamples}.
\end{proof}

\printbibliography
\end{document}